\documentclass[11pt]{article}
\usepackage[top=35truemm,bottom=40truemm,left=35truemm,right=35truemm]{geometry}
\usepackage{amsthm}
\usepackage{amsmath}
\usepackage{enumerate}  
\usepackage{amssymb}  
\usepackage[all]{xy} 

\usepackage{color}

\theoremstyle{definition} 
\newtheorem{dfn}{Definition}[section]
\newtheorem{thm}[dfn]{Theorem}
\newtheorem{lem}[dfn]{Lemma}
\newtheorem{prop}[dfn]{Proposition}
\newtheorem{cor}[dfn]{Corollary}
\newtheorem{ex}[dfn]{Example}
\newtheorem*{rem}{Remark}
\title{Thom form in equivariant \v{C}ech-de Rham theory}
\author{Ko Fujisawa}

\def\frakg{\mathfrak{g}}
\def\Ad{\mathrm{Ad}}
\def\coloneqq{\mathrel{\mathop:}=}
\def\edc#1#2{\Omega_{G}^{#1}(#2)}

\def\ech#1#2{H_{G}^{#1}(#2)}
\def\simto{\stackrel{\sim}{\rightarrow}}

\def\CdR#1{\Omega_{G}^{#1}(\mathcal{U})}
\def\CdRc#1{H_{G}^{#1}(\mathcal{U})}
\def\rCdR#1#2#3{\Omega_{G}^{#1}(\mathcal{#2},#3)}
\def\Eomega#1{\Omega^{#1}(M,E)}
\def\EGomega#1{\Omega_{G}^{#1}(M,E)}

\def\C{\mathbb{C}}

\begin{document}

\maketitle

\begin{abstract} 
In the present paper, we provide the foundation of a $G$-equivariant \v{C}ech-de Rham theory for a compact Lie group $G$ by using the  Cartan model of equivariant differential forms. 
Our approach is quite elementary without referring to the Mathai-Quillen framework. 
In particular, by a direct computation, 
we give an explicit formula of the $U(l)$-equivariant Thom form of $\C^l$, 
which deforms the classical Bochnor-Martinelli kernel. 
Also we discuss a version of equivariant Riemann-Roch formula. 
\end{abstract}

\section{Introduction}
As well known, the \v{C}ech-de Rham cohomology of a smooth manifold is a hypercohomology joining  
the \v{C}ech complex and the de Rham complex, which 
has been introduced for proving the equivalence between these two cohomology theories 
(cf. Bott-Tu \cite{bott1982differential}). 
Afterwards, Tatsuo Suwa has successfully established the \v{C}ech-de Rham theory  
as a tool for computing and describing explicit formulas at the level of cocycles; 
indeed, it yields several applications such as 
localization formulae of characteristic classes and 
index theorems of vector fields on possibly singular varieties 
(Suwa \cite{suwa1998indices, suwa2016complex}, Brasselet-Seade-Suwa \cite{brasselet2009vector}) 
and also index theorems for fixed points of holomorphic self-maps 
(Abate-Bracci-Tovena \cite{abate2004index}, Bracci-Suwa \cite{bracci2004residues}). 
In the present paper, we provide the foundation of a $G$-equivariant version of the \v{C}ech-de Rham theory 
for a compact Lie group $G$ by combining Suwa's construction with the classical Cartan model of equivariant differential forms.

Of our particular interest is to describe the equivariant characteristic classes and their localization at the level of cocycles in an explicit and constructible way. 
Let $M$ be a $G$-manifold and 
$\pi: E \to M$ a $G$-equivariant complex vector bundle of rank $l$ with the zero section $\Sigma \simeq M$. 
Put $\mathcal{W}=\{W_{0},W_{1}\}$ with $W_{0}=E\setminus \Sigma$ and $W_{1}=E$. 
The equivariant Thom form is simply given as an element of the relative equivariant \v{C}ech-de~Rham complex 
$$(0,\pi^{*}\varepsilon_{eq},-\psi_{eq})\in\rCdR{l}{W}{W_{0}},$$
where $\varepsilon_{eq}$ is the equivariant Euler form and $\psi_{eq}$ is the equivariant angular form 
such that $d_{eq}\psi_{eq}=-\pi^{*}\varepsilon_{eq}$ (Theorem \ref{equivariant Thom form}). 
A main result is an explicit expression of the universal equivariant Thom form 
for the trivial $U(l)$-equivariant bundle $\mathbb{C}^{l}\to \{0\}$, 
 that involves an $\mathfrak{u}(l)^*$-valued differential form 
 whose constant term is just the classical Bochner-Martinelli kernel 
 (Theorem \ref{equivariant Bochner-Martinelli kernel}). The equivariant Thom form of $E$ 
is now obtained from this universal form via the equivariant Chern-Weil map.  In our approach, it may be constructed via the localization of equivariant characteristic classes.
Indeed, the equivariant Thom class $\psi_{eq}^{E}$ is equal to 
the localized equivariant top Chern class with respect to the diagonal section $s_{\Delta}$:
$$\varPsi_{eq}^{E}=c_{\Sigma}^{l}({\pi^{*}E,s_{\Delta}})_{eq}$$
(Theorem \ref{equivariant Thom class}). 
Finally, we establish 
an essential version of equivariant Riemann-Roch theorem (Theorem \ref{RR}): 
$$ch_{\Sigma}^{*}(\lambda_{\pi^{*}E^{*}},s_{\Delta})_{eq}=\varPsi_{eq}^{E}\cdot td^{-1}(\pi^{*}E)_{eq}.$$

The most emphasized point is as follows.  In the theory of Mathai-Quillen \cite{mathai1986superconnections},
 the equivariant Thom form is introduced through the fermionic integral and supersymmetry arguments, 
and in this context, Paradan-Vergne \cite{paradan2007equivariant} 
described equivariant Thom forms for  oriented real vector bundles in several variants of de Rham complex. 
In contrast, our approach is quite elementary and simply minded -- 
basically we use only definite integrals for computations, without using the Mathai-Quillen framework. 
The present paper is the basis for further researches; 
for instance,  it is promising to study $\bar{\partial}$-Thom forms and Atiyah classes in equivariant \v{C}ech-Dolbeault theory in complex holomorphic context; 
also another equivariant \v{C}ech-de Rham theory can be considered  
using the Borel construction via the {\it simplicial method},  instead of using the Cartan model as above, 
that certainly leads to the de Rham theory for differentiable stacks. 
Those will be discussed in somewhere else. 
 
The present paper is organized as follows. 
In Section 1, after reviewing briefly the Cartan model, 
we describe the equivariant \v{C}ech-de Rham complex by following Suwa's construction. 
In Section 2, we then take up the equivariant Chern-Weil theory in our setting. 
In particular, we show that our localized equivariant top Chern form provides 
an explicit formula of the universal $U(l)$-equivariant Thom form. 
Finally, in Section 3, we see that our equivariant Thom form immediately 
leads an equivariant version of the Riemann-Roch theorem for 
the zero locus of a section of a complex vector bundle.

The author would like to thank his supervisor, Toru Ohmoto, for guiding him to this subject and many instructions, 
and is also grateful to Tatsuo Suwa for his interests and his warm encouragement.

\section{Equivariant \v{C}ech-de Rham cohomology}

\subsection{Equivariant de Rham cohomology}
\noindent Let $M$ be a smooth manifold and $G$ a Lie group with Lie algebra $\frakg$. We denote by $(\Omega^{*}(M),d)$ the $\mathbb{C}$-valued de Rham complex of $M$ and by $\mathbb{C}[\frakg]$ the algebra of polynomials on $\frakg$ (which is isomorphic to the symmetric algebra $S(\frakg^{*})$ of $\frakg^{*}$). Suppose that $G$ acts on $M$ smoothly. Then, for each element $X\in\frakg$, we obtain a vector field denoted by $X_M$:
$$X_{M}(m)=\left.\frac{d}{dt}\right|_{t=0}\exp(-tX)\cdot m$$

\noindent And, for $X\in\frakg$, we denote by $\iota_X$ the contraction with respect to $X_M$:
$$\iota_X\coloneqq \iota(X_M):\Omega^{k}(M)\to\Omega^{k-1}(M)$$

If $X_1,...,X_n$ is a basis of $\frakg$, we will let $x^1,...,x^n$ denote the corresponding dual basis. Then we naturally get the left action of $G$ on $\Omega^{*}(M)$ and $\mathbb{C}[\frakg]$ as follows: For $g\in G$,
$$ \omega\mapsto g\cdot\omega\coloneqq L_{g^{-1}}^{*}\omega, \ \ \omega\in\Omega^{*}(M)$$
$$ x^{I}\mapsto g\cdot x^{I}\coloneqq(\Ad^{*}_{g}x)^{I},\ \ x^{I}\in\mathbb{C}[\frakg]$$
where $L_{g^{-1}}^{*}$ is the pull back of a left transformation $L_{g^{-1}}$ and $\Ad^{*}_{g}$ is the coadjoint action of $G$ on $\frakg^{*}$ and $I=(i_1,...,i_n)$ is a multi-index.

\begin{dfn}
$\alpha=\sum_{I}x^{I}\otimes\omega_{I}\in\mathbb{C}[\frakg]\otimes\Omega^{*}(M)$ is called $G$-equivariant differential form, if it satisfies the following condition: For any $g\in G$
$$g\cdot\alpha\coloneqq \sum_{I}g\cdot x^{I}\otimes g\cdot \omega_{I}=\sum_{I}x^{I}\otimes\omega_{I}=\alpha$$
The wedge product of two equivariant forms is defined as the usual wedge product of differential forms. We denoted by $\Omega_{G}^{*}(M)\coloneqq (\mathbb{C}[\frakg]\otimes\Omega^{*}(M))^{G}$ the algebra of $G$-equivariant differential forms. The degree of an equivariant form $\alpha=x^{I}\otimes \omega_{I}\ (|I|=p,\omega_{I}\in\Omega^{k}(M))$ is defined by $\textrm{deg}(\alpha)\coloneqq 2p+k.$ The wedge product of two equivariant forms is defined as follows; for $X\in\frakg$,
$$(\alpha\wedge\beta)(X)\coloneqq\alpha(X)\wedge\beta(X),$$
where the wedge product on the right hand side is the usual wedge product of differential forms.
\end{dfn}

\begin{rem}
In other words, a $G$-equivariant differential form $\alpha=\sum_{I}x^{I}\otimes\omega_{I}$ may be also regarded as a $G$-equivariant polynomial map $\alpha:\frakg\mapsto \Omega^{*}(M)$, i.e.
$$\alpha(\sum \xi^{i}X_{i})=\sum \xi^{I}\omega_{I},\ \  \alpha(\Ad_{g}X)=L_{g^{-1}}^{*}\alpha=g\cdot\alpha(X)$$
$$\xymatrix{
{\frakg} \ar[r]^-{\alpha} \ar[d]_{Ad_{g}} & {\Omega^{*}(M)} \ar[d]^{L^{*}_{g^{-1}}} \\
{\frakg} \ar[r]_-{\alpha} & {{\Omega^{*}(M)}}
}$$
\end{rem}

\begin{dfn}
The \textit{twisted de Rham differential} $d_{eq}$ is defined as follows. For $\alpha\in\Omega_{G}^{*}(M)$ and $X\in \frakg$,
$$(d_{eq}\alpha)(X)\coloneqq d(\alpha(X))-\iota_{X}\alpha(X)$$
\end{dfn}

\noindent Then, it is easy to see $d_{eq}\circ d_{eq}=0$ and $(\edc{*}{M}, d_{eq})$ is a cochain complex (cf.\cite{libine2007lecture}). 

\begin{dfn}
The $p$-th \textit{equivariant de Rham cohomology algebra} is defined by the $p$-th cohomology of the $\mathbb{Z}$-graded complex $(\edc{*}{M},d_{eq})$:
$$H_{G}^{p}(M)\coloneqq \textrm{Ker}d^{p}_{eq}/\textrm{Im}d^{p-1}_{eq}$$
\end{dfn}

\begin{rem}
If a compact Lie group $G$ acts on $M$ freely, we have the following isomorphism;
$$H_{G}^{*}(M)\simto H^{*}(M/G)$$
where $H^{*}(M/G)$ is the de Rham cohomology of $M/G$. (cf.\cite{guillemin2013supersymmetry})
\end{rem}

\begin{prop}
Let $M,N$ be $G$-manifold. If $f:M\to N$ is $G$-morphism, then it induces a pull-back
$$f^{*}:\Omega_{G}^{*}(N)\to \Omega_{G}^{*}(M),\ \  x^{I}\otimes\omega_{I}\mapsto x^{I}\otimes f^{*}\omega_{I}$$
and it satisfies that $d_{eq}f^{*}=f^{*}d_{eq}$. Therefore, we get a homomorphism
$$f^{*}:\ech{*}{N}\to \ech{*}{M}$$
\end{prop}

\subsection{Equivariant \v{C}ech-de Rham cohomology}

Let $G$ be a compact Lie group and $M$ a $G$-manifold (i.e. a manifold given $G$-action). Let $\mathcal{U}=\{U_{\alpha}\}_{\alpha\in I}$ be a $G$-invariant open covering of $M$ (i.e. for any $g\in G,\ g\cdot U_{\alpha}=U_{\alpha}$). We assume that $I$ is an ordered set such that if $U_{\alpha_{0}\cdots\alpha_{r}}\coloneqq U_{\alpha_{0}}\cap\cdots\cap U_{\alpha_{r}}\neq\emptyset$, the induced order on the subset $\{\alpha_{0},...,\alpha_{r}\}$ is total. We set 
$$I^{(r)}=\{(\alpha_{0},...,\alpha_{r})\in I^{r+1}\mid \alpha_{0}<\cdots<\alpha_{r},\ \alpha_{\nu}\in I\}.$$

\begin{dfn}
We define $C^{p}(\mathcal{U},\Omega_{G}^{q})$ to be the direct product:
$$C^{p}(\mathcal{U},\Omega_{G}^{q})\coloneqq \prod_{(\alpha_{0},...,\alpha_{p})\in I^{(p)}}\Omega_{G}^{q}(U_{\alpha_{0}\cdots\alpha_{p}})$$
An element $\sigma\in C^{p}(\mathcal{U},\Omega_{G}^{q})$ assigns to each $(\alpha_{0},...,\alpha_{p})\in I^{(p)}$ a form $\sigma_{\alpha_{0}...\alpha_{p}} \in \Omega_{G}^{q}(U_{\alpha_{0}\cdots\alpha_{p}}).$
The coboundary operator
$$\delta:C^{p}(\mathcal{U},\Omega_{G}^{q})\to C^{p+1}(\mathcal{U},\Omega_{G}^{q})$$
is defined by
$$(\delta \sigma)_{\alpha_{0}...\alpha_{p+1}}\coloneqq\sum_{\nu=0}^{p+1}(-1)^{\nu}\sigma_{\alpha_{0}...\widehat{\alpha_{\nu}}...\alpha_{p+1}},$$
where  $\ \widehat{\ }\ $ means the letter under it is to be omitted and each form $\sigma_{\alpha_{0}...\widehat{\alpha_{\nu}}...\alpha_{p+1}}$ is to be restricted to $U_{\alpha_{0}\cdots\alpha_{p+1}}$. This together with the $G$-equivariant operator
$$d_{eq}:C^{p}(\mathcal{U},\Omega_{G}^{q})\to C^{p}(\mathcal{U}, \Omega_{G}^{q+1})$$
makes $C^{*}(\mathcal{U},\Omega_{G}^{*})$ a double complex.  Put 
$$\Omega_{G}^{r}(\mathcal{U})\coloneqq \bigoplus_{p+q=r}C^{p}(\mathcal{U},\Omega_{G}^{q})$$
and define for $p$-forms $\sigma \in C^{p}(\mathcal{U},\Omega_{G}^{q})$ 
$$D_{eq} \sigma \coloneqq \delta \sigma+ (-1)^p d_{eq} \sigma.$$

We call $(\Omega_{G}^{*}(\mathcal{U}),D_{eq})$ the equivariant \v{C}ech-de Rham complex and its $r$-th cohomology $H_{G}^{r}(\mathcal{U})$ the $r$-th equivariant \v{C}ech-de Rham cohomology of $\mathcal{U}$. 
\end{dfn}

\begin{thm}\label{natural isomorphism}
The natural homomorphism $r:\Omega^{r}_{G}(M)\to C^{0}(\mathcal{U},\Omega_{G}^{r})\subset \Omega_{G}^{r}(\mathcal{U})$ (which assigns to an $\omega\in\Omega^{p}_{G}(M)$ the cochain $\xi$ given by $\xi_{\alpha}=\omega|_{U_{\alpha}}$) induces an isomorphism:
$$r:H_{G}^{r}(M)\simto H^{r}_{G}(\mathcal{U})$$
\end{thm}

\begin{proof}

The same argument as for non equivariant case (Suwa \cite{suwa1998indices}) works. 
Here we use a $G$-equivariant partition of unity subordinate to the covering $\mathcal{U}$ 
(cf. Guillemin-Sternberg \cite{guillemin2002moment}). 
\end{proof}

The cup product of equivariant differential forms is also defined in the same way as in Suwa \cite{suwa2016complex}. 
In particular, it holds that 
$$D_{eq}(\xi\smile\eta)=D_{eq}\xi\smile\eta+(-1)^{r}\xi\smile D_{eq}\eta.$$

\begin{ex}\upshape 
(\textit{relative equivariant \v{C}ech-de Rham cohomology})  
Let $\mathcal{U}=\{U_{0},U_{1}\}$ be a $G$-invariant open covering of $M$. Then we have
$$\Omega_{G}^{r}(\mathcal{U})=\Omega_{G}^{r}(U_{0})\oplus\Omega_{G}^{r}(U_{1})\oplus\Omega_{G}^{r-1}(U_{01}).$$
The differential of an element $\xi=(\xi_{0},\xi_{1},\xi_{01})\in\Omega_{G}^{r}(\mathcal{U})$ is given by
$$D_{eq}\xi=(d_{eq}\xi_{0},\ d_{eq}\xi_{1},\ \xi_{1}-\xi_{0}-d_{eq}\xi_{01}).$$
Now we set
$$\rCdR{p}{U}{U_{0}}=\{\xi=(\xi_{0},\xi_{1},\xi_{01})\in\CdR{p}\mid \xi_{0}=0\},$$
which is a subcomplex of $(\CdR{*},D_{eq})$. 
Then its $p$-th cohomology is called the \textit{$p$-th relative equivariant \v{C}ech-de Rham cohomology of} $(\mathcal{U},U_{0})$ and we denote it by $H_{G}^{p}(\mathcal{U},U_{0})$.

In this case, the cup product $\CdR{r}\times\CdR{r'}\longrightarrow\CdR{r+r'}$ is defined by 
$$(\xi_{0},\xi_{1},\xi_{01})\smile(\eta_{0},\eta_{1},\eta_{01})=(\xi_{0}\wedge\eta_{0}, \xi_{1}\wedge\eta_{1}, (-1)^{r}\xi_{0}\wedge\eta_{01}+\xi_{01}\wedge\eta_{1}).$$
Putting $\xi_{0}=0$, we have a paring $\rCdR{r}{U}{U_{0}}\times\edc{r'}{U_{1}}\to\rCdR{r+r'}{U}{U_{0}}$
$$(0,\xi_{1},\xi_{01})\smile\eta_{1}=(0, \xi_{1}\wedge\eta_{1},\xi_{01}\wedge\eta_{1}).$$
\end{ex}

\subsection{Equivariant fiber integration and Thom form}
We follow the same argument as in Suwa \cite{suwa2016complex}. 
Hereafter, let $G$ be a compact Lie group.

\begin{dfn}
$\pi:T\to M$ is called an \textit{equivariant fiber bundle}, if $\pi:T\to M$ is a fiber bundle and $G$ acts on $T$ by bundle maps (in other words, $T$,$M$ are $G$-manifolds and $\pi:T\to M$ is $G$-morphism)
\end{dfn}

\begin{dfn}
Let $M$ be an oriented compact manidold and $\pi:T\to M$ be an equivariant oriented fiber bundle with fiber $F$ of dimension $l$, where $F$ is compact oriented possibly with boundary.  We define the $G$-equivariant fiber integration $\pi_{*}$ as follows;
$$(\pi_{*}\alpha)(X)\coloneqq \pi_{*}(\alpha(X)), \ \ \textrm{for}\  \alpha\in\edc{*}{T},X\in\frakg,$$
where $\pi_{*}$ on the right hand side is the usual fiber integration (Refer to \cite{suwa2016complex})
\end{dfn}

If $G$ acts on $T$ and $M$ preserving the orientations, directly computing, we see that $\pi_{*}(\edc{p}{T})\subset\edc{p-l}{M}$. Namely, we get a $\mathbb{C}$-linear map
$$\pi_{*}:\edc{p}{T}\to\edc{p-l}{M}$$

\begin{prop}\label{projection formula}In the above situation, the equivariant fiber integration has the following fundamental properties: 
\begin{enumerate}[(1)]
\item For $\alpha\in\edc{p}{T}$ and $\beta\in\edc{q}{M}$,
$$\pi_{*}(\alpha\wedge\pi^{*}\beta)=\pi_{*}\alpha\wedge\beta$$
\item Let $\partial T$ be a boundary of $T$ and $i:\partial T\hookrightarrow T$ be the inclusion. Then we have $$\pi_{*}\circ d_{eq}+(-1)^{l+1}d_{eq}\circ\pi_{*}=(\partial \pi)_{*}\circ i^{*}$$
\end{enumerate}
\end{prop}

\begin{proof}
It is shown in entirely the same way as the non equivariant case  \cite{suwa2016complex}
\cite{bott1972lectures}.
\end{proof}

In the following, we introduce the $G$-equivariant fiber integration on relative equivariant \v{C}ech-de Rham cochains. Let $\pi:E\to M$ be a $G$-equivariant oriented vector bundle of rank $l$ (that is, $\pi:E\to M$ is a vector bundle and $G$ acts on $E$ by vector bundle maps). We identify $M$ with the image of the zero section of $E$. Setting $W_{0}=E\setminus M$ and $W_{1}=E$, $\mathcal{W}=\{W_{0},W_{1}\}$ is a $G$-invariant open covering of $E$. Let $T_1\to M$ be a closed unit ball bundle in $W_1$ with respect to some $G$-invariant Riemannian metric on $E$ (since $G$ is a compact Lie group, it exists) and set $T_{0}=E\setminus \textrm{Int}T_{1}$. Then $\{T_{0},T_{1}\}$ is honeycomb system adapted to $\mathcal{W}$ (for details, see \cite{suwa2016complex}). Let $\pi_{1}$ and $\pi_{01}$ denote the restriction of $\pi$ to $T_{1}$ and $T_{01}$ respectively. Thus,
\begin{itemize}
\setlength{\itemsep}{-3pt}
\item $\pi_{1}:T_{1}\to M$ is a $G$-equivariant closed $l$-unit ball bundle
\item $\pi_{01}:T_{01}\to M$ is a $G$-equivariant $(l-1)$-sphere bundle.
\end{itemize}
By the definition of honeycomb system, the orientation of $T_{01}$ is opposite to that of the boundary $\partial T_{1}$ of $T_{1}$.

\begin{dfn}
The $G$-equivariant fiber integration on relative equivariant \v{C}ech-de~Rham cochains
$$\pi_{*}:\rCdR{p}{W}{W_{0}}\to\edc{p-l}{M}$$
is defined by
$$\pi_{*}\alpha\coloneqq (\pi_{1})_{*}\alpha_{1}+(\pi_{01})_{*}\alpha_{01},$$
where $\alpha=(0, \alpha_{1},\alpha_{01})\in\rCdR{p}{W}{W_{0}}$.
\end{dfn}

\begin{prop}
For $\alpha\in\rCdR{p}{W}{W_{0}},\beta\in\edc{q}{W_{1}}$, we have
$$\pi_{*}(\alpha\smile\pi^{*}\beta)=\pi_{*}\alpha\wedge\beta$$
\end{prop}

\begin{proof}
Take $\alpha=(0,\alpha_{1},\alpha_{01})\in\rCdR{p}{W}{W_{0}}$. By using Proposition \ref{projection formula},
\begin{eqnarray*}
\pi_{*}(\alpha\smile\pi^{*}\beta)&=&\pi_{*}(0,\alpha_{1}\wedge\pi^{*}\beta,\alpha_{01}\wedge\pi^{*}\beta) \\
&=&(\pi_{1})_{*}(\alpha_{1}\wedge\pi^{*}\beta)+(\pi_{01})_{*}(\alpha_{01}\wedge\pi^{*}\beta) \\
&=&(\pi_{1})_{*}\alpha_{1}\wedge\beta+(\pi_{01})_{*}\alpha_{01}\wedge\beta \\
&=&\pi_{*}\alpha\wedge\beta
\end{eqnarray*}
\end{proof}

\begin{prop}
In the above situation, we have the following formula;
$$\pi_{*}\circ D_{eq}+(-1)^{l+1}d_{eq}\circ \pi_{*}=0$$
Thus $\pi_{*}:\rCdR{p}{W}{W_{0}}\to\edc{p-l}{M}$ induces a homomorphism 
$\pi_{*}:H_{G}^{p}(\mathcal{W},W_{0})\to H_{G}^{p-l}(M)$.
\end{prop}

\begin{proof}
Applying Proposition \ref{projection formula} to $\pi_{1},\pi_{01}$ and noting that $(\partial \pi_{1})_{*}=-(\pi_{01})_{*}$, we obtain the above formula by directly computing.
\end{proof}

\noindent
The same Mayer-Vietoris argument in non-equivariant case (Theorem 5.3 in \cite{suwa1998indices}) shows that $\pi_{*}:H_{G}^{p}(\mathcal{W},W_{0})\to H_{G}^{p-l}(M)$ is isomorphism. Then, there exists the inverse map
$$(\pi_{*})^{-1}:H_{G}^{p-l}(M)\simto H_{G}^{p}(\mathcal{W},W_{0})$$
and we denote it by $T_{E}$ and call it the $G$-\textit{equivariant Thom isomorphism}. Then, setting
$$\varPsi^{E}_{eq}\coloneqq T_{E}([1])\in H_{G}^{l}(\mathcal{W},W_{0}) \ \ (1\in\edc{0}{M}),$$
we call it \textit{$G$-equivariant Thom class}. It follows from Proposition \ref{projection formula} that
$$\pi_{*}(T_{E}([1])\smile\pi^{*}\beta)=\pi_{*}T_{E}([1])\wedge\beta=\beta$$
$$\Longrightarrow T_{E}(\beta)=\varPsi^{E}_{eq}\smile\pi^{*}\beta$$
Then, we may take the following form as representative element of $\varPsi^{E}_{eq}$.

\begin{thm}\label{equivariant Thom form}
The equivariant Thom class $\varPsi^{E}_{eq}\in H_{G}^{l}(\mathcal{W},W_{0})$ is represented by the following form
$$(0,\pi^{*}\varepsilon_{eq},-\psi_{eq})\in\rCdR{l}{W}{W_{0}},$$
where $\varepsilon_{eq}$ is $d_{eq}$-closed $G$-equivariant $l$-form on $M$ and $\psi_{eq}$ is $G$-equivariant $(l-1)$-form on $W_{01}$ such that 
$$d_{eq}\psi_{eq}=-\pi^{*}\varepsilon_{eq}\ \ \textrm{in}\  W_{01}\ \textrm{and}\ -(\pi_{01})_{*}\psi_{eq}=1$$
\end{thm}

\begin{proof}
Suppose that $\varPsi^{E}_{eq}=[(0,\psi_{1},\psi_{01})] \in H_{G}^{l}(\mathcal{W},W_{0})$. Since
$$D_{eq}(0,\psi_{1},\psi_{01})=(0,d_{eq}\psi_{1},\psi_{1}-d_{eq}\psi_{01})=0,$$
we see that $\psi_{1}$ is closed $l$-form on $W_{1}$ and $\psi_{01}$ is $(l-1)$-form such that $d_{eq}\psi_{01}=\psi_{1}$ on $W_{01}$. Note that $\pi:E=W_{1}\to M$ induces an isomorphism equivariant de Rham cohomology, because $\pi$ is $G$-equivariant deformation retract. So there exists $\phi_{1}\in\edc{l-1}{W_{1}}$ and $\varepsilon_{eq}\in\edc{l}{M}$ such that 
$$\psi_{1}=\psi^{*}\varepsilon_{eq}+d_{eq}\phi_{1}.$$
Here, setting $\psi_{eq}\coloneqq -\psi_{01}+\phi_{1}$ which is $(l-1)$-form on $W_{01}$, we see that
$$(0,\psi_{1},\psi_{01})=(0,\pi^{*}\varepsilon_{eq},-\psi_{eq})+D_{eq}(0,\phi_{1},0).$$
Thus, $\psi_{E}$ is represented by $(0,\pi^{*}\varepsilon_{eq},-\psi_{eq})$. Then we have
$$d_{eq}\psi_{eq}=-d_{eq}\psi_{01}+d\phi_{1}=-\psi_{1}+(\psi_{1}-\pi^{*}\varepsilon_{eq})=-\pi^{*}\varepsilon.$$
Moreover, we have
\begin{eqnarray*}
(\pi_{1})_{*}\psi_{1}&=&(\pi_{1})_{*}\pi^{*}\varepsilon_{eq}+(pi_{1})_{*}d_{eq}\phi_{1} \\
&=&(\pi_{1})_{*}d_{eq}\phi_{1}=(\partial \pi_{1})_{*}\phi_{1}+(-1)^{l}d_{eq}(\pi_{1})_{*}\phi_{1} \\
&=&-(\pi_{01})_{*}\phi_{1}.
\end{eqnarray*}
From this and $\pi_{*}(0,\psi_{1},\psi_{01})=1$, it follow that
\begin{eqnarray*}
(\pi_{1})_{*}\psi_{1}+(\pi_{01})_{*}\psi_{01}=1&\iff& (\pi_{1})_{*}\psi_{1}+(\pi_{01})_{*}\phi_{1}-(\pi_{01})_{*}\psi_{eq}=1 \\
&\iff& -(\pi_{01})_{*}\psi_{eq}=1.
\end{eqnarray*}
\end{proof}

\begin{rem}
The form $\varepsilon_{eq}$ above is called a \textit{$G$-equivariant Euler form} and $\psi_{eq}$ is called a \textit{$G$-equivariant global angular form}.
\end{rem}

\section{Equivariant Chern-Weil theory and Localization}
\subsection{Equivariant Chern-Weil theory}
Let $G$ be a compact Lie group and $\pi:E\to M$ a complex $G$-equivariant vector bundle of rank $l$.  We denote by $\Eomega{*}$ the set of $E$-valued differential forms on $M$. Then we define the set of $E$-valued $G$-equivariant differential forms on $M$ by
$$\Omega_{G}^{*}(M,E)\coloneqq (\mathbb{C}[\frakg]\otimes\Eomega{*})^{G},$$
where $G$ acts on the section of $E$ such that for $g\in G$ and $s\in \Eomega{0}$
$$(g\cdot s)(m)\coloneqq g\cdot s(g^{-1}\cdot m).$$
Note that $\Omega_{G}^{*}(M,E)$ is the $\Omega_{G}^{0}(M)$-module.

\begin{dfn}\ \vspace{-3mm}
\begin{enumerate}[(1)]
\item A connection $\nabla:\Eomega{*}\to\Eomega{*+1}$ is called a $G$-\textit{invariant connection}, if $\nabla$ commutes with $G$-action on $\Eomega{*}$, that is, $g\nabla=\nabla g;$
$$\xymatrix{
{\Eomega{*}} \ar[r]^-{\nabla} \ar[d]_{g} & {\Eomega{*+1}} \ar[d]^{g} \\
{\Eomega{*}} \ar[r]_-{\nabla} & {\Eomega{*+1}}.
}$$

\item The \textit{equivariant connection} $\nabla_{eq}$ corresponding to a $G$-invariant connection $\nabla$ is the operator on $\mathbb{C}[\frakg]\otimes\Omega^{*}(M,E)$ defined by the formula: for $X\in\frakg,\ \alpha\in\mathbb{C}[\frakg]\otimes\Omega^{*}(M,E)$,
$$(\nabla_{eq}\alpha)(X)\coloneqq (\nabla-\iota_{X})\alpha(X).$$
\end{enumerate}
\end{dfn}

\begin{lem}
If $\alpha\in\EGomega{*}$, then $\nabla_{eq}\alpha\in\EGomega{*+1}.$
\end{lem}

\begin{proof}
Using $\alpha(\Ad_{g}X)=g\cdot\alpha(X)$,\ $g\nabla=\nabla g$ and $\iota_{\Ad_{g}X}=g\cdot\iota_{X}\cdot g^{-1}$, we easily see that $(\nabla_{eq}\alpha)(\Ad_{g}X)=g\cdot(\nabla_{eq}\alpha)(X)$
\end{proof}

\begin{dfn}
The equivariant curvature $K_{eq}:\EGomega{0}\to\EGomega{2}$ of an equivariant connection $\nabla_{eq}$ is defined by the formula: 
$$K_{eq}(X)\coloneqq\nabla_{eq}(X)^{2}+L_{X}^{E},$$
where $L_{X}^{E}$ is an infitesimal action of $\frakg$ induced by the $G$-action on $\EGomega{0}.$
\end{dfn}

\begin{lem}\label{independent}
For $f\in \Omega_{G}^{0}(M)$ and $s\in\EGomega{0}$, we have
$$K_{eq}(fs)=fK_{eq}(s).$$
Thus $K_{eq}$ is an element of $\Omega_{G}^{2}(M,End(E))$.
\end{lem}

In the following, we define $G$-equivariant characteristic classes.
Let $\nabla$ be a $G$-invariant connection, $\nabla_{eq}$ its equivariant connection and $K_{eq}$ its equivariant curvature as above. Take a $G$-invariant open set $U$ in $M$ such that $E$ is trivial on $U$. If $s^{(l)}=(s_1,...,s_l)$ is a local  frame of $E$ on $U$, we may write, for $i=1,...,l$ and $X\in\frakg$
\begin{eqnarray*}
(\nabla_{eq}s_{i})(X)&=&\sum_{j=1}^{l}\theta_{ji}\otimes s_{j},\ \  \theta_{ij}\in\Omega^{1}(U),\\
(K_{eq}s_{i})(X)&=&\sum_{j=1}^{l}\kappa_{ji}(X)\otimes s_{j}, \ \ \kappa_{ij}\in(\C[\frakg]\otimes\Omega^{0}(U))\oplus\Omega^{2}(U). 
\end{eqnarray*}
We call $\theta=(\theta_{ij})$ the connection matrix and $\kappa=(\kappa_{ij})$ the equivariant curvature matrix with respect to $s^{(l)}$. 
From the definition, $\kappa_{ij}$ is computed explicitly as follows. 
Letting $L_{X}^{E}s_{i}=\sum_{j=1}^{l}\ell_{ji}(X)s_j$, we have
\begin{eqnarray*}
(K_{eq}s_{i})(X)&=&\nabla^{2}s_{i}-\iota_{X}\nabla s_{i}+L_{X}^{E}s_i \\
&=&\sum_{j=1}^{l}\{d\theta_{ji}+\sum_{k=1}^{l}\theta_{jk}\wedge\theta_{ki}-\iota_{X}\theta_{ji}+\ell_{ji}(X)\}\otimes s_{j}
\end{eqnarray*}
and thus 
$$
\kappa_{ij}(X)=d\theta_{ij}+\sum_{k=1}^{l}\theta_{ik}\wedge\theta_{kj}-\iota_{X}\theta_{ij}+\ell_{ji}(X).
\eqno{(\kappa)}.$$
We will use this equality in the proof of Theorem \ref{equivariant universal Thom class} later. 
Moreover, this leads to an equivariant version of well-known Bianchi identity. 
For completeness, we prove it: 

\begin{lem}[equivariant Bianchi identity]
It holds that $d_{eq}\kappa=[\kappa,\theta]$. 
\end{lem}

\begin{proof}
Noting that $L_{X}^{E}\nabla s_i=\nabla L_{X}^{E}s_i$ (since $\nabla$ is $G$-invariant) and comparing the both sides locally, we have 
$$-L_{X}\theta_{ji}+d\ell(X)_{ji}=\sum_{k=1}^l(\ell_{jk}(X)\theta_{ki}-\theta_{jk}\ell_{ki}(X))$$
Thus,
\begin{eqnarray*}
d_{eq}\kappa_{ij}(X)&=&d\kappa_{ij}(X)-\iota_{X}\kappa_{ij}(X) \\
&=&\sum_{k=1}^{l}(d\theta_{ik}\wedge\theta_{kj}-\theta_{ik}\wedge d\theta_{kj})
-\sum_{k=1}^{l}(\iota_{X}(\theta_{ik})\theta_{kj}-\iota_{X}(\theta_{kj})\theta_{ik})\\
&&-d\iota_{X}\theta_{ij}-\iota_{X}d\theta_{ij}+d\ell_{ji}(X) \\
&=&\sum_{k=1}^{l}\{(d\theta_{ik}-\iota_{X}\theta_{ik})\wedge\theta_{kj}-\theta_{ik}\wedge (d\theta_{kj}-\iota_{X}\theta_{kj})\}-L_{X}\theta_{ij}+d\ell_{ji}(X) \\
&=&\sum_{k=1}^{l}\{(d_{eq}\theta_{ik}+\ell_{ik}(X))\wedge\theta_{kj}-\theta_{ik}\wedge (d_{eq}\theta_{kj}+\ell_{kj}(X))\}
\end{eqnarray*}
Therefore, letting $\ell(X)=(\ell_{ij}(X))$, the above equation may be written in terms of a matrix form as
\begin{eqnarray*}
d_{eq}\kappa&=&(d_{eq}\theta+\ell(X))\wedge\theta-\theta\wedge (d_{eq}\theta+\ell(X)) \\
&=&(d_{eq}\theta+\theta\wedge\theta+\ell(X))\wedge\theta-\theta\wedge(d_{eq}\theta+\theta\wedge\theta+\ell(X)) \\
&=&\kappa\wedge\theta-\theta\wedge\kappa=[\kappa,\theta].
\end{eqnarray*}
\end{proof}

\noindent
For a homogeneous invariant polynomial $\phi$ (that is, $\phi\in\mathbb{C}[\frak{gl}(l,\mathbb{C})]^{GL(l,\mathbb{C})}$), the $G$-equivariant characteristic form is defined by
$$\phi(\nabla_{eq})\coloneqq \phi(\kappa).$$ 
Then, it follows from the equivariant Bianchi identity that $\phi(\nabla_{eq})$ is $d_{eq}$-closed and
this is independent of the choice of a local frame of $E$ (Lemma \ref{independent}). 
The \textit{$G$-equivariant characteristic class} of $E$ for an invariant polynomial $\phi$ 
is defined by 
$$\phi(E)_{eq}\coloneqq [\phi(\nabla_{eq})]\in\ech{*}{M}.$$
In fact, this class is independent of the choice of $\nabla_{eq}$ (see below). 

Now, we switch to the setting of equivariant \v{C}ech-de Rham cohomology. 
We need the following $G$-equivariant Bott-difference form: 

\begin{prop}[Bott's difference form]\label{Bott's difference form}
Suppose $\nabla_{eq}^{(0)},...,\nabla_{eq}^{(p)}$ are $G$-equivariant connections for $E$. For a homogeneous invariant polynomial $\phi$ of degree $k$, there is a form $\phi(\nabla_{eq}^{(0)},...,\nabla_{eq}^{(p)})\in\edc{2k-p}{M}$ satisfying the following properties:
\begin{enumerate}[(1)]
\item $\phi(\nabla_{eq}^{(0)},...,\nabla_{eq}^{(p)})$ is alternating in the $p+1$ entries
\item $\sum_{\nu=0}^{p}(-1)^{\nu}\phi(\nabla_{eq}^{(0)},...,\widehat\nabla_{eq}^{(\nu)},...,\nabla_{eq}^{(p)})+(-1)^{p}d_{eq}\phi(\nabla_{eq}^{(0)},...,\nabla_{eq}^{(p)})=0$
\end{enumerate}
We call $\phi(\nabla_{eq}^{(0)},...,\nabla_{eq}^{(p)})$ a $G$-equivariant Bott-difference form with respect to $G$-equivariant connections $\nabla_{eq}^{(0)},...,\nabla_{eq}^{(p)}$.
\end{prop}

\begin{proof}
Let $\rho: \mathbb{R}^{p}\times M\to M$ be the natural projection, where $G$ acts on $\mathbb{R}^{p}$ trivially. Then, we define a $G$-equivariant connection $\tilde\nabla_{eq}$ for $\rho^{*}E$ by
$$\tilde\nabla_{eq}=(1-\sum_{\nu=1}^{p}t_{\nu})\rho^{*}\nabla_{eq}^{(0)}+\sum_{\nu=1}^{p}t_{\nu}\rho^{*}\nabla_{eq}^{(\nu)},$$
Letting $\rho^{\prime}:\Delta^{p}\times M\to M$ be the restriction of $\rho$, we get the fiber integration
$$\rho_{*}^{\prime}:\edc{*}{\Delta^{p}\times M}\to\edc{*-p}{M},$$
where $\Delta^{p}$ is the standard $p$-simplex.
And, setting
$$\phi(\nabla_{eq}^{(0)},...,\nabla_{eq}^{(p)})\coloneqq \rho^{\prime}_{*}\phi(\tilde\nabla),$$
we have the desired form satisfying (1), (2) (Use the Stokes theorem and the formula Proposition \ref{projection formula} (2)). 
\end{proof}

Let $\mathcal{U}=\{U_{\alpha}\}_{\alpha\in I}$ be a $G$-invariant open covering of $M$ as in section 1.2. Let $\pi:E\to M$ be a complex vector bundle of rank $l$ and $\phi$ an invariant polynomial homogeneous of degree $k$. For each $\alpha$, we choose a connection $\nabla_{eq}^{(\alpha)}$ for $E|_{U_{\alpha}}$ and for the collection $\nabla_{eq}^{*}=(\nabla_{eq}^{(\alpha)})_{\alpha\in I}$, we define $\phi(\nabla_{eq}^{*})\in\CdR{2k}$ by
$$\phi(\nabla_{eq}^{*})_{\alpha_{0}\cdots\alpha_{p}}\coloneqq \phi(\nabla_{eq}^{(0)},...,\nabla_{eq}^{(p)})\in \edc{2k-p}{U_{\alpha_{0}\cdots\alpha_{p}}}$$

\begin{lem}
In the above situation, we have the followings.
\begin{enumerate}[(1)]
\item $D_{eq}\phi(\nabla_{eq}^{*})=0$
\item  For another collection $\tilde\nabla_{eq}^{*}=(\tilde\nabla_{eq}^{(\alpha)})_{\alpha\in I}$, there exists the element $\psi\in\CdR{2k-1}$ such that
$$\phi(\tilde\nabla_{eq}^{*})-\phi(\nabla_{eq}^{*})=D_{eq}\psi.$$
\end{enumerate}
\end{lem}

\begin{proof}
(1) By direct computations.
(2) Setting
$$\psi=\sum_{\nu=0}^{p}\phi(\nabla_{eq}^{(\alpha_{0})},...,\nabla_{eq}^{(\alpha_{\nu})},\tilde\nabla_{eq}^{(\alpha_{\nu})},...,\tilde\nabla_{eq}^{(\alpha_{p})}),$$
we easily see that $\phi(\tilde\nabla_{eq}^{*})-\phi(\nabla_{eq}^{*})=D_{eq}\psi$.
\end{proof}

\noindent It follows from this lemma that the element $\phi(\nabla_{eq}^{*})$ defines a cohomology class $[\phi(\nabla_{eq}^{*})]\in\CdRc{2k}$ which depends only on $E$ but not on the choice of the collection of connections $\nabla_{eq}^{*}$. Also, from the following theorem, we may naturally regard $[\phi(\nabla_{eq}^{*})]$ as a characteristic class in $\CdRc{*}$.

\begin{thm}
The class $[\phi(\nabla_{eq}^{*})]$ in $\CdRc{*}$ corresponds to the class $\phi(E)_{eq}$ in $H_{G}^{*}(M)$ under the isomorphism of Theorem \ref{natural isomorphism}
\end{thm}

\begin{proof}
Take an equivariant connection $\nabla_{eq}$ on $M$. For each $U_{\alpha}\in\mathcal{U}$, defining $\nabla_{eq}^{(\alpha)}$ to be $\nabla_{eq}|_{U_{\alpha}}$, we see that it is an equivariant connection for $E|_{U_{\alpha}}$. Then for the collection $\nabla_{eq}^{*}=(\nabla_{eq}^{(\alpha)})_{\alpha\in I}$, by definition, 
$$\phi(\nabla_{eq}^{*})\in C^{0}(\mathcal{U},\Omega_{G}^{*}).$$
Thus, $r(\phi(\nabla_{eq}))=\phi(\nabla_{eq}^{*})$ and $r([\phi(\nabla_{eq}^{*})])=\phi(E)_{eq}.$
\end{proof}

As usual,  the total equivariant Chern form is given by
$$c^{*}(\nabla_{eq})\coloneqq \textrm{det}(I_{r}+\frac{\sqrt{-1}}{2\pi}\kappa)=1+c_{eq}^{i}(\nabla_{eq})+\cdots+c_{eq}^{l}(\nabla_{eq})\in\edc{*}{M}$$
and the total equivariant Chern class of $E$ is defined by its cohomology class
$$c_{eq}^{*}(E)=1+c_{eq}^{1}(E)+\cdots+c_{eq}^{l}(E)\in H_{G}^{*}(M).$$
Note that the form $c^{*}(\nabla_{eq})$ and the class $c_{eq}^{*}(E)$ is invertible in $\edc{*}{M}$ and $ H_{G}^{*}(M)$ respectively. 
In the same way as the non equivariant case, the equivariant Chern form (or class) has functoriality with respect to a pull-back and additivity with respect to an exact sequence.

\subsection{Localized equivariant characteristic classes}

Let $M$ be a $G$-manifold and $S$ a $G$-invariant closed set in $M$ and $\pi:E\to M$ a complex $G$-equivariant vector bundle of rank $l$. Letting $U_{0}=M\setminus S$ and $U_{1}$ a $G$-invariant neighborhood of $S$, we consider the $G$-invariant covering $\mathcal{U}=\{U_{0},U_{1}\}$. In what follows, let $H_{G}(\mathcal{U},U_{0})$ denote $H_{G}(M,M\setminus S)$.

Suppose there is some ``geometric object" $\gamma$ on $U_{0}$, to which is associated a class $\mathcal{C}$ of equivariant connections for $E$ on $U_{0}$ such that, for a certain homogeneous invariant polynomial $\phi$,
$$\phi((\nabla_{eq}^{(0)})_{0},...,(\nabla_{eq}^{(k)})_{0})\equiv 0\ \ \textrm{if every}\ (\nabla_{eq}^{(i)})_{0}\ \textrm{belongs}\ \mathcal{C}.$$
A equivariant connection $(\nabla_{eq})_{0}$ for $E$ on $U_{0}$ is said to be \textit{special}, if $(\nabla_{eq})_{0}$ belongs to $\mathcal{C}$ and the polynomial $\phi$ as above is said to be \textit{adapted} to $\gamma$.

\begin{lem}\label{localization lemme}
In the above situation, suppose that $\nabla_{eq}^{0}$ is special and $\phi$ is adapted to $\gamma$. 
The class of 
$$\phi(\nabla_{eq}^{*})=(0,\phi(\nabla_{eq}^{1}),\phi(\nabla_{eq}^{0},\nabla_{eq}^{1}))\in\Omega^{*}_{G}(\mathcal{U},U_{0})$$
is independent of the choice of the special equivariant connection $\nabla_{eq}^{0}$ or the equivariant connection $\nabla_{eq}^{1}$.
\end{lem}

\begin{proof}
If $\nabla_{eq}^{0}$ and ${\nabla'}_{eq}^{0}$ are both special, by using $\phi(\nabla_{eq}^{0},{\nabla'}_{eq}^{0})=0$ and Proposition \ref{Bott's difference form} , we have
$$(0,\phi(\nabla_{eq}^{1}),\phi({\nabla'}_{eq}^{0},\nabla_{eq}^{1}))-(0,\phi(\nabla_{eq}^{1}),\phi(\nabla_{eq}^{0},\nabla_{eq}^{1}))=D_{eq}(0,0,\phi(\nabla_{eq}^{0},{\nabla'}_{eq}^{0},\nabla_{eq}^{1})).$$
Similarly, for equivariant connections $\nabla_{eq}^{1}$ and ${\nabla'}_{eq}^{1}$ on $U_{1}$,
$$(0,\phi({\nabla'}_{eq}^{1}),\phi(\nabla_{eq}^{0},{\nabla'}_{eq}^{1}))-(0,\phi(\nabla_{eq}^{1}),\phi(\nabla_{eq}^{0},\nabla_{eq}^{1}))=D_{eq}(0,\phi({\nabla'}_{eq}^{1},\nabla_{eq}^{1}),\phi(\nabla_{eq}^{0},\nabla_{eq}^{1},{\nabla'}_{eq}^{1})).$$
\end{proof}

\noindent From this, we may define the following.

\begin{dfn}
If $(\nabla_{eq})_{0}$ is special and $\phi$ (homogeneous of degree $d$) is adapted to $\gamma$, the class $\phi_{S}(E,\gamma)\in H_{G}^{2d}(M,M\setminus S)$ is defined by 
$$\phi_{S}(E,\gamma)\coloneqq [\phi(\nabla_{eq}^{*})]$$
and is called the \textit{localized equivariant characteristic class} of $\phi(E)_{eq}$ at $S$ by $\gamma$.
\end{dfn}

In the following, we consider a geometric object by frames and its localized equivariant Chern class.
Suppose $\pi:E\to M$ is a complex $G$-equivariant vector bundle of rank $l$. Then, it follows from a way of definition of the $G$-equivariant Bott-difference form that for any $k$ equivariant connections $\nabla_{eq}^{(1)},...,\nabla_{eq}^{(k)}$ for $E$,
$$c^{i}(\nabla_{eq}^{(1)},...,\nabla_{eq}^{(k)})\equiv 0\ \ \ \textrm{for}\ \ i\geq l+1$$

\noindent As a consequence, we have the following.

\begin{lem}\label{vanishing theorem by frame}
Let $s^{(r)}=(s_1,...,s_r)$ be an $r$-frame of $E$ on a $G$-invariant open set $U$ in $M$. If $\nabla_{eq}^{(1)},...,\nabla_{eq}^{(k)}$ is $s^{(r)}$-trivial on $U$, then on $U$
$$c^{i}(\nabla_{eq}^{(1)},...,\nabla_{eq}^{(k)})\equiv 0\ \ \ \textrm{for}\ i\geq l-r+1.$$
\end{lem}

\noindent From this, we have the following.

\begin{dfn}
Let $s^{(r)}$ be a local frame on $M\setminus S$. If $\nabla_{eq}^{0}$ is $s^{(r)}$-trivial, by Lemma \ref{vanishing theorem by frame}, 
$$c^{i}(\nabla_{eq}^{*})=(0,c^{i}(\nabla_{eq}^{1}),c^{i}(\nabla_{eq}^{0},\nabla_{eq}^{1}))\ \ \textrm{for}\ i\geq l-r+1$$
and induce the class $[c^{i}(\nabla_{eq}^{*})]\in H^{2i}_{G}(M,M\setminus S)$. Since this class is independent of the choice of $s^{(r)}$-trivial $G$-equivariant connection $\nabla_{eq}^{0}$ on $U_{0}$ and a $G$-equivariant connection $\nabla_{eq}^{1}$ on $U_{1}$ by Lemma \ref{localization lemme}, we denote by
$$c_{S}^{i}(E,s^{(r)})_{eq}\coloneqq [c^{i}(\nabla_{eq}^{*})]$$
and we call it \textit{the localized Chern class} of $c^{i}(E)_{eq}$ by $s^{(r)}$ at $S$.
\end{dfn}

\subsection{Equivariant Thom class via localized Chern class}

Suppose the unitary group $U(l)$ ($\mathfrak{u}(l)$ is the Lie algebra of $U(l)$) acts on $\mathbb{C}^{l}$ naturally. Then 
$$\pi:\mathbb{C}^{l}\to\{0\}$$
 is clearly an $U(l)$-equivariant vector bundle. Setting
 $$W_{0}=\C^l\setminus\{0\},\ \ W_{1}=\C^l,$$
 we have an $U(l)$-invariant covering $\mathcal{W}=\{W_{0},W_{1}\}$.
 We consider the pull-back of $\C^{l}$ by $\pi$, i.e.,
$$\pi^{*}\C^{l}=\{(z_{1},z_{2})\in \C^l\times \C^l\mid \pi(z_{1})=\pi(z_{2})\}=\C^{l}\times\C^{l}$$
$$\varpi:\pi^{*}\C^{l}=\C^{l}\times\C^{l}\to \C^l,$$
where $\varpi$ is the projection to the second factor. From the definition of pull-back, $U(l)$ acts on $\C^l\times\C^l$ diagonaly ($A(z_1,z_2)=(Az_1,Az_2)$) and $\varpi:\pi^{*}\C^{l}=\C^{l}\times\C^{l}\to \C^l$ is an $U(l)$-equivariant vector bundle. Then the diagonal section
$$s_{\Delta}:\C^l\to \pi^{*}\C^l=\C^l\times\C^l,\ \ z\mapsto (z,z)$$
is naturally $U(l)$-invariant frame on $\C^l\setminus\{0\}$. Thus, we may consider the localized Chern class of $c^{l}(\pi^{*}\C^l)_{eq}$ by $s_{\Delta}$, that is,
$$c^{l}(\pi^{*}\C^l,s_{\Delta})_{eq}\in H^{2l}_{U(l)}(\C^l, \C^l\setminus\{0\})$$
This class is represented by the following form
$$(0, c^{l}_{eq}(D_{eq}^{1}), c^{l}_{eq}(D_{eq}^{0},D_{eq}^{1}))\in \Omega^{2l}_{U(l)}(\mathcal{W},W_{0}),$$
where
\begin{itemize}
\item $D_{eq}^{0}$ is an $s_{\Delta}$-trivial $U(l)-$equivariant connection for $\pi^{*}\C^l$ on $W_{0}=\C^l\setminus\{0\},$
\item $D_{eq}^{1}$ is an $U(l)$-equivariant connection for $\pi^{*}\C^l$ on $W_{1}=\C^l.$
\end{itemize}

On the other hand, as a real vector bundle of rank $2l$, we may consider the $U(l)$-equivariant Thom class $\varPsi_{eq}^{\C^l}\in H_{U(l)}^{2l}(\C^l,\C^l\setminus\{0\})$. 

\begin{thm}\label{equivariant universal Thom class}[equivariant universal Thom class]
In the above situation, we have
$$c^{l}(\pi^{*}\C^l,s_{\Delta})_{eq}=\varPsi_{eq}^{\C^l}.$$
\end{thm} 
\begin{proof}
Setting 
$$T_{1}=D^{2l}=\{z\in \C^l\mid \|z\|\leq1\},\ \ T_{0}=\C^l\setminus\textrm{Int}T_{1},$$
we have a honeycomb system $\{T_{0},T_{1}\}$ adapted to $\mathcal{W}$. Note that $T_{01}=-S^{2l-1}.$
By the definition of $\psi_{eq}^{\C^l}$, it suffices to find the equivariant connections $D_{eq}^{0}, D_{eq}^{1}$ satisfying
$$(\pi_{1})_{*}c^{l}(D_{eq}^{1})+(\pi_{01})_{*}c^{l}(D_{eq}^{0},D_{eq}^{1})=1,$$
where $\pi_{1}:T_{1}\to\{0\}$, $\pi_{01}:T_{01}\to\{0\}$. Let sections $s_{1},...,s_{l}$ of $\varpi:\pi^{*}\C^l\to\C^l$ be 
$$s_{i}(z)=(e_{i},z)\ \ (i=1,...,l),$$
where $\{e_{i}\}$ is the standard basis of $\C^l$. Now we define the connection $D_{1}$ for $\pi^{*}\C^l$ on $\C^l$ by
$$D_{1}(\sum_{i=1}^{l}f_{i}s_{i})\coloneqq \sum_{i=1}^{l}df_{i}\otimes s_{i}\ \ \ (\textrm{for}\ f_{i}\in C^{\infty}(\C^l)),$$
which is $s^{l}=(s_{1},...,s_{l})$-trivial. Also, we easily see that $D_{1}$ is a $U(l)$-invariant connection. Thus we may define the equivariant connection $D_{eq}^{1}$ corresponding to $D_{1}$. From the definition of $D_{eq}^{1}$, the form degree of its curvature form is $0$ and $(\pi_{1})_{*}c^{l}(D_{eq}^{1})=0$. 

Next we define the connection $D_{0}$ for $\pi^{*}\C^l$ on $\C^l\setminus\{0\}$ by
$$D_{0}s_{i}=-\frac{\bar{z_{i}}}{\|z\|^2}\sum_{j=1}^{l}dz_j\otimes s_j\ \ (i=1,...,l)$$
For $f_{i}\in C^{\infty}(\C^l\setminus\{0\}) (i=1,...,l),\ \ g\in U(l)$, we have
$$g\cdot D_{0}(\sum_{i=1}^{l}f_{i}s_{i})=\sum_{i=1}^{l}((g\cdot df_{i})\otimes (g\cdot s_{i})+(g\cdot f_{i})(g\cdot D_{0}s_{i}))$$
$$D_{0}(g\cdot (\sum_{i=1}^{l}f_{i}s_{i}))=\sum_{i=1}^{l}((g\cdot df_{i})\otimes (g\cdot s_{i})+(g\cdot f_{i})(D_{0}(g\cdot s_{i}))).$$
Therefore, to show that $D_{0}$ is $U(l)$-invariant connection, it suffices to check 
$$g\cdot (D_{0}s_{i})=D_{0}(g\cdot s_{i}).$$
For $g=(g_{ij})\in U(l)$, directly computing, we have
$$g\cdot\left(-\frac{\bar{z_{i}}}{\|z\|^2}\right)=-\sum_{k=1}^{l}g_{ki}\frac{\bar{z_{k}}}{\|z\|^2}$$
$$g\cdot \textrm{d}z_{j}=\sum_{m=1}^{l}\overline{g_{mj}}dz_{m},\ \ g\cdot s_{j}=\sum_{n=1}^{l}g_{nj}s_n.$$
Thus, we have
\begin{eqnarray*}
g\cdot (D_{0}s_{i})&=&-\sum_{k=1}^{l}g_{ki}\frac{\bar{z_{k}}}{\|z\|^2}\sum_{j=1}^{l}\left\{(\sum_{m=1}^{l}\overline{g_{mj}}dz_{m})\otimes (\sum_{n=1}^{l}g_{nj}s_n)\right\} \\
&=&-\sum_{k=1}^{l}g_{ki}\frac{\bar{z_{k}}}{\|z\|^2}\sum_{m=1}^{l}\sum_{n=1}^{l}\left(\sum_{j=1}^{l}\overline{g_{mj}}g_{nj}\right)(dz_m\otimes s_n) \\
&=&-\sum_{k=1}^{l}g_{ki}\frac{\bar{z_{k}}}{\|z\|^2}\sum_{m=1}^{l}\sum_{n=1}^{l}\delta_{mn}(dz_m\otimes s_n)\\
&=&-\sum_{k=1}^{l}g_{ki}\frac{\bar{z_{k}}}{\|z\|^2}\sum_{j=1}^{l}dz_j\otimes s_j
\end{eqnarray*}
and
$$D_{0}(g\cdot s_i)=D_{0}(\sum_{k=1}^{l}g_{ki}s_k)=\sum_{k=1}^{l}g_{ki}D_{0}s_{k}=-\sum_{k=1}^{l}g_{ki}\frac{\bar{z_{k}}}{\|z\|^2}\sum_{j=1}^{l}dz_j\otimes s_j.$$
So we get $g\cdot (D_{0}s_{i})=D_{0}(g\cdot s_{i})$. Also, for the diagonal section $s_{\Delta}=\sum_{i=1}^{l}z_{i}s_{i}$, we easily see that $D_{0}s_{\Delta}=0$. Hence we have an $s_{\Delta}$-trivial $U(l)-$equivariant connection $D_{eq}^{0}$ corresponding to $D_{0}$ for $\pi^{*}\C^l$ on $W_{0}=\C^l\setminus\{0\}$. The rest of proof is to show that 
\begin{equation}\label{int1}
-\int_{S^{2l-1}}c^{l}(D_{eq}^{0},D_{eq}^{1})=1. 
\end{equation}
The connection matrix $\theta_{0}=(\theta_{ij})$ of $D_{eq}^{0}$ with respect to $(s_{1},...,s_{l})$ is express by 
$$\theta_{ij}=-\frac{\bar{z_{j}}}{\|z\|^2}dz_i,$$
while the connection matrix $\theta_{1}$ of $D_{eq}^{1}$ with respect to $(s_{1},...,s_{l})$ is zero. For $t\in\mathbb{R}$ and the natural projection $\rho:\mathbb{R}\times(\mathbb{C}^{l}\setminus\{0\})\to\mathbb{C}^{l}\setminus\{0\}$, we set
$$\tilde{D}_{eq}=(1-t)\rho^{*}D^{0}_{eq}+t\rho^{*}D^{1}_{eq},$$
and denote $\rho^{*}D^{0}_{eq}, \rho^{*}D^{1}_{eq}$ by $D^{0}_{eq}, D^{1}_{eq}$ for short. Then the connection matrix $\tilde{\theta}$ of $\tilde{D}_{eq}$ with respect to $(s_{1},...,s_{l})$ is given by $\tilde{\theta}=(1-t)\theta_{0}$,  and thus by ($\kappa$) in subsection 2.1, the corresponding equivariant curvature matrix $\tilde{\kappa}$ is given by 
$$\tilde{\kappa}(X)=d\tilde{\theta}+\tilde{\theta}\wedge \tilde{\theta} - \iota_X \tilde{\theta} + \ell(X)$$
for $X=(X_{ij})\in\mathfrak{u}(l)$. 
Recall that $\ell(X)=(\ell_{ij}(X))_{ij}$ is defined by $L_{X}^{E}s_i=\sum_{j=1}^{l}\ell_{ji}(X)s_j$.  
For later use, we rewrite it as 
$$\tilde{\kappa}(X)=-dt\wedge\theta_{0}+\kappa_{t}(X),$$ 
$$\kappa_{t}(X)=(1-t)d\theta_{0}+(1-t)^{2}\theta_{0}\wedge\theta_{0}-(1-t)\iota_{X}\theta_{0}+\ell(X).$$
By the definition of the equivariant Bott-difference form, 
\begin{eqnarray*}
c^{l}(D^{0}_{eq},D^{1}_{eq})&=&\rho^{\prime}_{*}c^{l}(\tilde{\kappa})\\
&=&\left(\frac{\sqrt{-1}}{2\pi}\right)^{l}\rho^{\prime}_{*}\textrm{det}\tilde{\kappa}\\
&=&-\left(\frac{\sqrt{-1}}{2\pi}\right)^l\sum_{j=1}^{l}\int_{0}^{1}\textrm{det}Q_{j}dt,
\end{eqnarray*}
where $Q_j$ is the matrix obtained from $\kappa_{t}$ by replacing the $j$-th column by that of $\theta_{0}$. In the following, we compute $\textrm{det}Q_{j}$. Computing the $(i,j)$-entry of $d\theta_{0}$, $\theta_{0}\wedge\theta_{0}$, $\ell(X)$ and $\iota_{X}\theta_{0}$, we have
$$(d\theta_{0})_{ij}=-\left(\frac{1}{\|z\|^{2}}d\bar{z_{j}}+\bar{z_{j}}d\left(\frac{1}{\|z\|^2}\right)\right)\wedge dz_{i}$$
$$(\theta_{0}\wedge\theta_{0})_{ij}=\frac{1}{\|z\|^{4}}\bar{z_{j}}dz_{i}\wedge(\sum_{k=1}^{l}\bar{z_{k}}dz_{k})$$
$$(\iota_{X}\theta_{0})=-\frac{\bar{z_{j}}}{\|z\|^{2}}\sum_{k=1}^{l}X_{ik}z_{k}.$$
$$\ell(X)_{ij}=X_{ij}.$$
We set the matrices $\tau(X)$ and $\eta(X)$ as follows;
\begin{eqnarray*}\label{bm}
\tau(X)_{ij}&\coloneqq&-\frac{(1-t)}{\|z\|^{2}}d\bar{z}_{j}\wedge dz_{i}+X_{ij} \\
\eta(X)_{ij}&\coloneqq&-(1-t)\bar{z}_{j}d\left(\frac{1}{\|z\|^{2}}\right)\wedge dz_{i}+(1-t)^2(\theta_{0}\wedge\theta_{0})_{ij}+(1-t)(\iota_{X}\theta_{0})_{ij}.
\end{eqnarray*}
Then, $\kappa_{t}(X)=\tau(X)+\eta(X)$. Denoting $k$-th column of $\kappa_{t}(X)$ and $\tau(X), \eta(X)$ by $\kappa_{t}(X)^{(k)}$ and $\tau(X)^{(k)}, \eta(X)^{(k)}$ respectively, the matrix $Q_{j}$ may be expressed as follows;
$$\textrm{det}Q_{j}=\textrm{det}[\kappa_{t}(X)^{(1)},...,\theta_{0}^{(j)},...,\kappa_{t}(X)^{(l)}].$$ We decompose this determinant with respect to the columns $\tau(X)^{(k)}, \eta(X)^{(k)}$ by using multilinearity of determinant. Note that, if more than two columns of $\eta(X)$ appear in the determinant obtained from the decomposed term, the term vanishes. Thus, we have
$$\textrm{det}Q_{j}=\textrm{det}R_{j}+\sum_{k\neq j}\textrm{det}R_{jk},$$
where
$$R_{j}\coloneqq[\tau(X)^{(1)},...,\theta_{0}^{(j)},...,\tau(X)^{(l)}]$$
$$R_{jk}\coloneqq[\tau(X)^{(1)},...,\theta_{0}^{(j)},...,\eta(X)^{(k)},...,\tau(X)^{(l)}].$$
By the definition, we see that $\textrm{det}R_{jk}=-\textrm{det}R_{kj}$. Directly computing, we have
\begin{eqnarray}
c^{l}(D^{0}_{eq},D^{1}_{eq})&=&-\left(\frac{\sqrt{-1}}{2\pi}\right)^{l}\left\{\sum_{j=1}^{l}\int_{0}^{1}\textrm{det}R_{j}dt+\int_{0}^{1}\sum_{j=1}^{l}\sum_{k\neq j}\textrm{det}R_{jk}dt\right\}
\nonumber\\
&=&-\left(\frac{\sqrt{-1}}{2\pi}\right)^{l}\sum_{j=1}^{l}\int_{0}^{1}\textrm{det}R_{j}dt 
\nonumber\\
&=&-C_l\frac{\sum_{j=1}^{l}\overline{\Phi_j(z)}\wedge\Phi(z)}{||z||^{2l}}+\textrm{(terms with $X_{ij}$)},
\end{eqnarray}
where \begin{eqnarray*}
\Phi(z)&=&dz_1\wedge\cdots dz_l \\
\Phi_i(z)&=&(-1)^{i-1}z_i dz_i\wedge\cdots\wedge \widehat{dz_i}\wedge\cdots\wedge dz_l
\end{eqnarray*}
and
$$C_l=(-1)^{\frac{l(l-1)}{2}}\frac{(l-1)!}{(2\pi\sqrt{-1})^l}.$$
Thus, we have $-\int_{S^{2l-1}}c^{l}(D_{eq}^{0},D_{eq}^{1})=1$, since $C_l\frac{\sum_{j=1}^{l}\overline{\Phi_j(z)}\wedge\Phi(z)}{||z||^{2l}}$ coinsides the Bochner-Martinelli kernel $\beta_{l}$ on $\mathbb{C}^{l}$(see \cite{suwa1998indices}).
\end{proof}

\subsection{Explicit formula of universal $U(l)$-equivariant Thom form}

We give an explicit formula of universal $U(l)$-equivariant Thom form 
$$(0, c^{l}_{eq}(D_{eq}^{1}), c^{l}_{eq}(D_{eq}^{0},D_{eq}^{1}))\in \Omega^{2l}_{U(l)}(\mathcal{W},W_{0}).$$
In particular,  higher terms in (\ref{bm}) are precisely determined.  

We provide some notations to simplify a calculation. 
Let $V$ be a complex vector space of dimension $l$ with a  basis $e_{1}, \cdots ,e_{l}$.
For any anticommutative $\mathbb{Z}$-graded algebra $\mathcal{A}$, we consider the algebra $\mathcal{A}\otimes \wedge^{*}V$ with the following wedge product; $(\alpha\otimes\xi)\wedge(\beta\otimes\eta)\coloneqq (\alpha\wedge\beta)\otimes (\xi\wedge\eta)$. 
It is easy to see the following lemma: 

\begin{lem}\label{lemA}\ 
\begin{enumerate}[(1)]
\item Let $\omega_{i}=\sum_{k=1}^{l}\omega_{ik}\otimes e_{k}$, then 
$$\omega_{1}\wedge\cdots\wedge\omega_{l}
=\sum_{\sigma\in\mathfrak{S}_{l}}\, \textrm{sgn}(\sigma)(\omega_{1\sigma(1)}\wedge\cdots\wedge\omega_{l\sigma(l)})\otimes (e_{1}\wedge\cdots\wedge e_{l})$$ 
\item Let $\alpha=\sum_{k=1}^{l}\alpha_{k}\otimes e_{k}$ 
and $\beta=\sum_{k=1}^{l}\beta_{k}\otimes e_{k}\in\mathcal{A}\otimes \wedge^{*}V$ with $\deg(\alpha_{k})=s$ and $\deg(\beta_{k})=t$, then 
$$\alpha\wedge\beta=-(-1)^{st}\beta\wedge\alpha$$
\end{enumerate}
\end{lem}

We write  $[l]:=\{1, 2, \cdots, l\}$. 
If $I$ is a subset of $[l]$, 
we denote by $e_{I}$ the product 
$e_{i_{1}}\wedge\cdots\wedge e_{i_{p}}$ 
where we write 
$I=\{i_{1},i_{2},...,i_{p}\}$ with $i_{1}<i_{2}<\cdots<i_{p}$. 
Denote by $|I|=p$, the cardinality of $I$. 
For $I=\{i_{1},i_{2},...,i_{p}\}$ and $I'=\{i'_{1},i'_{2},...,i'_{p}\}$ in $[l]$, 
we set 
$\epsilon_{(I, I')}:=(-1)^{\sum_{s=1}^p (i_s+i'_s)}$. 
Let $X=[X_{ij}]\in\mathfrak{u}(l)$,  and 
denote by $X_{I, I'}$ the retainer minor of $X$ with respect to $I$ and $I'$: 
$X_{I,I^{\prime}}=\det [X_{i_s i'_t}]_{1\le s, t \le p}$. 
If $1\le k \le l$ and $k \not\in J$, 
we denote by $\epsilon(k,J)$ the sign such that 
$e_{k}\wedge e_{J}=\epsilon(\{k\},J)e_{\{k\}\cup J}$. Put 
$$\gamma_{(k,I,J)}=(-1)^{\frac{|J|(|J|-1)}{2}} \left(\frac{\sqrt{-1}}{2\pi}\right)^{l} |J|! \,\, \epsilon({k},J). $$

\begin{thm}\label{equivariant Bochner-Martinelli kernel}
For $X\in\mathfrak{u}(l)$, we have
$$\chi_{eq}(X)\coloneqq c^{l}(D_{eq}^{1})=\left(\frac{\sqrt{-1}}{2\pi}\right)^{l}\textrm{det}X$$
$$\beta_{eq}(X)\coloneqq c^{l}(D_{eq}^{0},D_{eq}^{1})
=\sum_{k,I,J}\gamma_{(k,I,J)}\sum_{I^{\prime},J^{\prime}}\epsilon_{(I, I')} X_{I,I^{\prime}}\frac{\bar{z}_{k}d\bar{z}_{J}\wedge dz_{J^{\prime}}}{\|z\|^{2(|J|+1)}}$$
where for $1\leq k \leq l$, the sets $I, J$ vary over the subsets of $[l]$ such
that $\{k\}\cup I\cup J$ is a partition of $[l]$, 
 and $I^{\prime}$ and $J^{\prime}$ vary over the subsets of $[l]$ 
 such that $|I|=|I^{\prime}|$ and $I^{\prime}\cup J^{\prime}$ is a partition of $[l]$. 
\end{thm}

\begin{proof}
Let $\theta_{1}$ and $\kappa_{1}$ be the connection matrix and the corresponding equivariant curvature matrix with respect to the frame $(s_{1},...,s_{l})$. Since $\theta_{1}=0$, 
$$\kappa_{1}(X)=d\theta_{1}+\theta_{1}\wedge\theta_{1}-\iota_{X}\theta_{1}+\ell(X)=\ell(X)=X.$$
Thus, $c^{l}(D_{eq}^{1})=c^{l}(X)=\left(\frac{\sqrt{-1}}{2\pi}\right)^{l}\textrm{det}X$. Next, we compute $c^{l}(D_{eq}^{0},D_{eq}^{1})$. Set $Y=-\frac{\|z\|^{2}}{1-t}X$. Then
\begin{eqnarray*}
\textrm{det}R_{k}&=&
\det \, [\tau(X)^{(1)},...,\theta_{0}^{(k)},...,\tau(X)^{(l)}]\\
&=&(-1)^{l}\frac{(1-t)^{l-1}}{\|z\|^{2l}} P 
\end{eqnarray*}
with 
\begin{eqnarray*}
P&=&\sum_{\sigma\in\mathfrak{S}_{l}}\, \textrm{sgn}(\sigma)(d\bar{z_1}\wedge dz_{\sigma(1)}+Y_{1\sigma(1)})\wedge\cdots \cdots\wedge\bar{z_k}dz_{\sigma(k)}\wedge\cdots\wedge(d\bar{z_l}\wedge dz_{\sigma(l)}+Y_{l\sigma(l)}) \\
&=&\sum_{\sigma\in\mathfrak{S}_{l}}\, \textrm{sgn}(\sigma)\sum_{I,J}(d\bar{z}_{j_{1}}\wedge dz_{\sigma(j_{1})})\wedge\cdots\wedge(\bar{z}_{k}dz_{\sigma(k)})\wedge\cdots\wedge(d\bar{z}_{j_{q}}\wedge dz_{\sigma(j_{q})})Y_{i_{1}\sigma(i_{1})}\cdots Y_{i_{p}\sigma(i_{p})}\\
&=&\sum_{I,J}\, (-1)^{\frac{|J|(|J|-1)}{2}}\epsilon({k},J)\cdot \bar{z}_{k}d\bar{z}_{J}\cdot \\
&&\qquad \qquad 
\sum_{\sigma\in\mathfrak{S}_{l}}\, \textrm{sgn}(\sigma) dz_{\sigma(j_1)}\wedge\cdots\wedge dz_{\sigma(k)}\wedge \cdots\wedge dz_{\sigma(j_q)}Y_{i_{1}\sigma(i_{1})}\cdots Y_{i_{p}\sigma(i_{p})}, 
\end{eqnarray*}
where $\{k\}$, $I=\{i_{1},i_{2},...,i_{p}\}$ and $J=\{j_{1},j_{2},...,j_{q}\}$ is a partition of $[l]$.
Set 
$$Z=\sum_{i=1}^{l}dz_{i}\otimes e_{i},\ \ Y_{j}=\sum_{i=1}^{l}Y_{ji}\otimes e_{i}.$$
By Lemma \ref{lemA}, we see 
\begin{eqnarray*}
&&\sum_{\sigma\in\mathfrak{S}_{l}}\, 
\textrm{sgn}(\sigma)\, dz_{\sigma(j_1)}\wedge\cdots\wedge dz_{\sigma(k)}\wedge \cdots\wedge dz_{\sigma(j_q)}Y_{i_{1}\sigma(i_{1})}\cdots Y_{i_{p}\sigma(i_{p})}\otimes(e_{1}\wedge\cdots\wedge e_{l})
\\
&=&
Z\wedge\cdots\wedge Y_{i_1}\wedge\cdots\wedge Z \wedge\cdots\wedge Y_{i_p}\wedge\cdots\wedge Z \\
&=&(-1)^m(Y_{i_1}\wedge\cdots\wedge Y_{i_p})\wedge (Z\wedge\cdots\wedge Z) 
\end{eqnarray*}
where $m=\sum_{s=1}^{p}i_{s}-\frac{1}{2}p(p+1)$. 
Note that 
$$Y_{i_1}\wedge\cdots\wedge Y_{i_p}
=\sum_{I^{\prime}}Y_{I,I^{\prime}}\, e_{I^{\prime}}, 
\;\; Z\wedge\cdots\wedge Z=\sum_{J^{\prime}}(q+1)! \, dz_{J^{\prime}}\otimes e_{J^{\prime}},$$
where 
$I'$ runs over subsets of $p$ elements in $[l]$, 
$J'$ runs over subsets of $(l-p)$ elements in $[l]$, and 
$Y_{I,I^{\prime}}$ is a retainer minor of $[Y_{ij}]$ with respect to $I$ and $I^{\prime}$. 
Then 
\begin{eqnarray*}
&&(-1)^m(Y_{i_1}\wedge\cdots\wedge Y_{i_p})\wedge (Z\wedge\cdots\wedge Z) \\
&=&(-1)^{m}\sum_{I^{\prime}} \, Y_{I,I^{\prime}} \sum_{J^{\prime}}\, (q+1)!\, dz_{J^{\prime}}\otimes(e_{I^{\prime}}\wedge e_{J^{\prime}})\\
&=&(q+1)!\sum_{I^\prime, J^\prime} (-1)^{\sum_{s=1}^{p}(i_{s}+i^{\prime}_{s})}Y_{I,I^{\prime}}dz_{J^{\prime}}\otimes(e_{1}\wedge\cdots\wedge e_{l}).
\end{eqnarray*}
Since $Y_{ii'}=-\frac{\|z\|^{2}}{1-t}X_{ii'}$, 
we have 
$$\textrm{det}R_{k}=\sum_{I,J}\epsilon({k},J)(-1)^{\frac{|J|(|J|-1)}{2}+l}(|J|+1)!\sum_{I^{\prime},J^{\prime}}\epsilon_{(I, I')} X_{I,I^{\prime}}\frac{\bar{z}_{k}d\bar{z}_{J}\wedge dz_{J^{\prime}}}{\|z\|^{2(|J|+1)}}(1-t)^{|J|}.$$
Hence, 
\begin{eqnarray*}
c^{l}(D_{eq}^{0},D_{eq}^{1})
&=&-\left(\frac{\sqrt{-1}}{2\pi}\right)^{l}\sum_{k=1}^{l}\int_{0}^{1}\textrm{det}R_{k}\, dt \\
&=&\sum_{k,I,J}\gamma_{(k,I,J)}\sum_{I^{\prime},J^{\prime}} \epsilon_{(I, I')} X_{I,I^{\prime}}\frac{\bar{z}_{k}d\bar{z}_{J}\wedge dz_{J^{\prime}}}{\|z\|^{2(|J|+1)}}.
\end{eqnarray*}
\end{proof}

\begin{ex}\upshape
For small $l$, the equivariant Bochner-Martinelli kernel is computed as follows. 
\begin{enumerate}
\item 
In the case of $l=1$,
$$\beta_{eq}(X)=\frac{\sqrt{-1}}{2\pi}\frac{{\bar{z}}dz}{\|z\|^2}=\frac{\sqrt{-1}}{2\pi}\frac{dz}{z}.$$
This is nothing but the original (non-equivariant) kernel. 
\item 
In the case of $l=2$, for $X=(X_{ij})\in \mathfrak{u}(2)$,
\begin{eqnarray*}
\beta_{eq}(X)&=&\left(\frac{\sqrt{-1}}{2\pi}\right)^{2}\left\{
\frac{{\bar{z}}_{1}d\bar{z}_{2}\wedge dz_{1} \wedge dz_{2}}{\|z\|^4}
-\frac{{\bar{z}}_{2}d\bar{z}_{1}\wedge dz_{1} \wedge dz_{2}}{\|z\|^4}\right. \\
&&\left. +X_{1,1}\frac{{\bar{z}}_{2}dz_{2}}{\|z\|^2}-X_{1,2}\frac{{\bar{z}}_{2}dz_{1}}{\|z\|^2}
-X_{2,1}\frac{{\bar{z}}_{1}dz_{2}}{\|z\|^2}+X_{2,2}\frac{{\bar{z}}_{1}dz_{1}}{\|z\|^2}\right\}.
\end{eqnarray*}
To be more specific, we see the real part of this form: 
Set $z_1=x_1+\sqrt{-1}y_1$, $z_2=x_2+\sqrt{-1}y_2$ and 
$$
X = \left(
    \begin{array}{cc}
     \sqrt{-1}A & B+\sqrt{-1}C \\
-B+\sqrt{-1}C & \sqrt{-1}D
    \end{array}
  \right),$$
where $A$,$B$,$C$,$D$ are real numbers. 
 Then, 
a simple computation shows 
\begin{eqnarray*}
&&Re(\beta_{eq}(X))\\
&&=\frac{1}{2\pi^2\|z\|^4}
(x_1dx_2\wedge dy_1\wedge dy_2 +x_2dx_1\wedge dy_1\wedge dy_2 \\
&&\ \ \ \ \ \ \ \ \ \ \ \ -y_1dx_1\wedge dx_2 \wedge dy_2-y_2dx_1\wedge dx_2\wedge dy_1) \\
&&+\frac{1}{4\pi^2\|z\|^2}
(-Ax_2dy_2+Ay_2dx_2+Bx_1dx_2+By_1dy_2-Bx_2dx_1-By_2dy_1\\
&&\ \ \ \ \ \ \ \ \ \ \ \ +Cx_1dy_2-Cy_1dx_2+Cx_2dy_1-Cy_2dx_1-Dx_1dy_1+Dy_2dx_2).
\end{eqnarray*}
This form coincides with the angular form for $\mathfrak{so}(4)$ 
 of Proposition 4.10 in Paradan-Vergne \cite{paradan2007equivariant}.

\end{enumerate}
\end{ex}

\subsection{Explicit formula of $G$-equivariant Thom form}
 In this subsection, applying the equivariant Chern-Weil map 
\cite{berline1992heat, paradan2007equivariant} 
to Theorem \ref{equivariant universal Thom class}, 
we obtain a formula expressing the equivariant Thom form for general $G$-vector bundles.

\begin{dfn}
Let $M$ be a manifold with a Lie group $G$-action. 
$\alpha\in \Omega^{*}(M)$ is called \textit{horizontal} if $\iota_{X}\alpha=0$ for any $X\in \frakg$. We denote by $\Omega^{*}(M)_{hor}$ the subalgebla formed by the differential form that are horizontal. Also we define the algebra of the basic differential forms as follows; 
$$\Omega^{*}(M)_{basic}\coloneqq (\Omega^{*}(M)_{hor})^{G}$$
\end{dfn}

Let $\pi:P\to B$ be a principal $G$-bundle. And suppose $G$ acts on a manifold $F$. For the associated bundle $\mathcal{F}=P\times_{G}F$, The Chern-Weil map in non-equivariant case gives the following isomorphism;
$$\phi_{\theta}^{F}:\Omega_{G}^{*}(F)\simto \Omega^{*}(\mathcal{F}),$$
where $\theta$ is a connection form of $P$. In more details, for a $G$-equivariant form $\alpha$, $\phi_{\theta}^{F}(\alpha)$ is equal to the projection of $\alpha(\Omega)\in \Omega(P\times F)^{G}$ on the basic space $\Omega(P\times F)_{basic}\simto \Omega(\mathcal{F})$, where $\Omega$ is the curvature of the connection $\theta$.
We give the equivariant version of this construction in the following.

Let $K$ and $G$ be two compact Lie groups and $P$ be a smooth manifold.  We assume that $K\times G$ acts on $P$ as follows; $(k,g)(y)\coloneqq kyg^{-1}$, for $k\in K,\ g\in G$. And $G$ acts on $P$ freely. Then, $B=P/G$ is a manifold provided with a left action of $K$. There is $K$-invariant connection $\theta$ of $P$, since $K$ is compact. Then, for a $K$-invariant connection $\theta$, $K$-equivariant curvature of $P$ is defined as follows;
$$\tilde{\Omega}\coloneqq d_{K}\theta+\frac{1}{2}[\theta\wedge\theta],$$
where $d_{K}$ is $K$-equivariant differential. Using this, we consider the equivariant Chern-Weil map;
$$\phi_{\theta}^{F}:\Omega_{G}^{*}(F)\simto \Omega_{K}^{*}(\mathcal{F}).$$
It is defined as follows. For a $G$-equivatiant form $\alpha$ on $F$, $\phi_{\theta}^{F}(\alpha)$ is equal to the projection of $\alpha(\tilde{\Omega})\in \Omega_{K}^{*}(P\times F)^{G}$ onto the basic space $\Omega_{K}^{*}(P\times F)_{basicG}\simto \Omega_{K}^{*}(\mathcal{F}).$

\begin{prop}\label{cochain map}
The equivariant Chern-Weil map above satisfies the following condition;
$$\phi_{\theta}^{F}\circ d_{G}=d_{K}\circ \phi_{\theta}^{F}.$$
\end{prop}

We construct the explicit formulas of $G$-equivariant Thom form in the following. First, we consider a $G$-equivariant vector bundle $\pi:E\to M$ and take a $G$-invariant metric for $E$. Then, for any $x\in M$, set $P_{x}=\{\xi:\mathbb{C}^{l}\to E_{x}:\textrm{isometry}\}$ and $P=\cup_{x\in M}P_{x}$ is naturally $U(l)$-equivariant $G$-principal bundle. The above argument applying for this, we get the following Chern-Weil maps;
$$\phi_{\theta}^{\mathbb{C}^{l}}:\Omega_{U(l)}^{*}(\mathbb{C}^{l})\simto \Omega_{G}^{*}(E)$$
$$\phi_{\theta}^{\mathbb{C}^{l}\setminus\{0\}}:\Omega_{U(l)}^{*}(\mathbb{C}^{l}\setminus\{0\})\simto \Omega_{G}^{*}(E\setminus \Sigma)$$
By using this, we may give the $G$-equivariant Thom form as follows;
$$(0, \phi_{\theta}^{\mathbb{C}^{l}}c^{l}_{eq}(D_{eq}^{1}), \phi_{\theta}^{\mathbb{C}^{l}\setminus\{0\}}c^{l}_{eq}(D_{eq}^{0},D_{eq}^{1}))\in \Omega^{2l}_{G}(\mathcal{W},W_{0})$$
It follow from Proposition \ref{cochain map} that this form is closed. Then, we denote by $c_{\Sigma}^{l}({\pi^{*}E,s_{\Delta}})_{eq}$ the class of this form, where $s_{\Delta}:E\to \pi^{*}E$ is the diagonal section and $\Sigma$ is the zero section of $E$. It is not difficult to show that the equivariant fiber integration is compatible with the equivariant Chern-Weil map. Thus, we have the following formula: 

\begin{thm}\label{equivariant Thom class}
In the above situation, we have
$$\varPsi_{eq}^{E}=c_{\Sigma}^{l}({\pi^{*}E,s_{\Delta}})_{eq},$$
where $\varPsi_{eq}^{E}$ is the $G$-equivariant Thom class for $E$.
\end{thm}

\section{Equivariant Riemann-Roch Theorem}

In this last section, we show 
 a version of equivariant Riemann-Roch theorem in our  setting. 
Indeed, it is entirely parallel to the description in non-equivariant case (cf. \cite{brasselet2009vector}\cite{suwa2000characteristic}). 

\subsection{Chern character and Todd class}
Let $G$ be a compact manifold and $E\to M$ be a $G$-equivariant vector bundle of rank $l$.
For a $G$-equivariant connection$\nabla_{eq}$ for $E$, let $K_{eq}$ denote its curvature and set $A=({\sqrt{-1}}/{2\pi})K_{eq}$.

\

For $G$-equivariant connection $\nabla_{eq}$, the equivariant Chern character form and Todd form is defined as follows;
$$\textrm{ch}^{*}(\nabla_{eq})\coloneqq \textrm{tr}(e^{A})$$
$$\textrm{td}(\nabla_{eq})\coloneqq \textrm{det}\left(\frac{A}{I-e^{-A}}\right)$$
Note that $I-e^{-A}$ is divisible by $A$ and the result is invertible so that
$$\textrm{td}^{-1}(\nabla_{eq})=\textrm{det}\left(\frac{I-e^{-A}}{A}\right)$$
In the same way of the Chern form, we may easily show that these form is closed and the classes of these form is independent of the choice of equivariant connections. Note that the constant term in $\textrm{td}(\nabla_{eq})$ is $1$ and that $\textrm{td}(\nabla_{eq})$ can be expressed as a series in $c^{i}(\nabla_{eq})$. Then, we have the following formula;
$$\sum_{i=0}^{l}(-1)^{i}\textrm{ch}^{*}({\bigwedge}^{i}\nabla_{eq}^{*})=c^{l}(\nabla_{eq})\cdot \textrm{td}(\nabla_{eq})^{-1},$$
where $\nabla_{eq}^{*}$ denotes the connection for $E^{*}$ dual to $\nabla_{eq}$ and $\bigwedge^{i}\nabla_{eq}^{*}$ the connection for $\bigwedge^{i}E^{*}$ induced by $\nabla_{eq}^{*}$. Here we set $\bigwedge^{0}E=\mathbb{C}\times M$ and $\bigwedge^{0}\nabla_{eq}^{*}=d_{eq}$, the twisted de Rham differential.

\subsection{Equivariant characteristic forms for virtual bundles}

Let $E_i$ $(i=0,...,q)$ be $G$-equivariant complex vector bundles. We may consider the virtual bundle
$\xi=\sum_{i=0}^{q}(-1)^{i}E_i$ (as an element of $K$-group of $G$-equivariant vector bundles on $M$) and a family of equivariant connections $\nabla^{\bullet}_{eq}=(\nabla^{(0)}_{eq},...,\nabla^{(q)}_{eq})$, where $\nabla^{(0)}_{eq}$ is a $G$-equivariant connection for $E_i$. We set 
$$c^{*}(\nabla_{eq}^{\bullet})=\prod_{i=0}^{q}c^{*}(\nabla_{eq}^{(i)})^{\epsilon(i)}\ \ \ \textrm{and}\ \ \  \textrm{ch}^{*}(\nabla_{eq}^{\bullet})=\sum_{i=0}^{q}(-1)^{i}\textrm{ch}(\nabla_{eq}^{(i)}),$$
where $\epsilon(i)=(-1)^{i}$. In general, for a symmetric series, we may define a form $\phi(\nabla_{eq}^{\bullet})$. It is closed and its class $\phi(\xi)$ is in $H_{G}^{*}(M)$. For two families of connections $(\nabla^{\bullet}_{eq})_{\nu}=((\nabla^{(0)}_{eq})_{\nu},...,(\nabla^{(q)}_{eq})_{\nu})$, $\nu=1,2$, the same argument for non-virtual version may define the Bott difference form $\phi((\nabla^{\bullet}_{eq})_0,(\nabla^{\bullet}_{eq})_1)$. From this, in the same way of non-virtual version, we easily see that $\phi(\xi)=[\phi(\nabla_{eq}^{\bullet})]$ is independent of the choice of a families of connections. 

We may also define the equivariant characteristic classes for virtual bundle in the equivariant \v{C}ech-de Rham cohomology as in section 2.1. It is sufficient to consider coverings $\mathcal{U}$ consisting of two open sets $U_0$ and $U_1$ for the sake of argument in the following. Then, taking a family of connections $(\nabla^{\bullet}_{eq})_{\nu}=((\nabla^{(0)}_{eq})_{\nu},...,(\nabla^{(q)}_{eq})_{\nu})$ for $\xi$ on each $U_{\nu}$, $\nu=0,1$, for the collection $(\nabla_{eq}^{\bullet})_{\star}=((\nabla_{eq}^{\bullet})_0, (\nabla_{eq}^{\bullet})_1)$, a cochain $\phi((\nabla_{eq}^{\bullet})_{\star})$ in $\Omega_{G}^{*}(\mathcal{U})$ is defined as follows;
$$\phi^{i}((\nabla_{eq}^{\bullet})_{\star})=(\phi^{i}((\nabla_{eq}^{\bullet})_0), \phi^{i}((\nabla_{eq}^{\bullet})_1), \phi^{i}((\nabla^{\bullet}_{eq})_0,(\nabla^{\bullet}_{eq})_1))$$
It is in fact a cocycle and defines a class $[\phi^{i}((\nabla_{eq}^{\bullet})_{\star})]$ in $H_{G}^{*}(\mathcal{U})$. It does not depend on the choice of the collection of families of connections $(\nabla_{eq}^{\bullet})_{\star}$ and corresponds to the class $\phi(\xi)$ under the isomorphism $H_{G}^{*}(\mathcal{U})\simeq H_{G}^{*}(M)$.

\subsection{Equivariant Riemann-Roch Theorem}

Let $M$ be as above, $s$ a $G$-invariant section in $M$. Let $S$ denote the zero set of $s$ (note that $S$ is also $G$-invariant). Letting $U_{0}=M\setminus S$ and $U_1$ a $G$-invariant neighborhood of $S$, we consider the $G$-invariant covering $\mathcal{U}=\{U_0,U_1\}$. We set $\lambda_{E^{*}}=\sum_{i=0}^{l}(-1)^{i}\bigwedge^{i}E^{*}$. Let $\nabla_0$ be an $s$-trivial $G$-equivariant connection for $E$ on $U_0$ and $\nabla_1$ an arbitrary $G$-equivariant connection for $E$ on $U_1$. Consider the Koszul complex associated to $(E,s)$ (for more details, see \cite{suwa2016complex});
$$\xymatrix{
0 \ar[r] & {\bigwedge}^{l}E^{*} \ar[r]^{d_{s}} & \cdots \ar[r]^{d_{s}} & {\bigwedge}^{1}E^{*} \ar[r]^{d_{s}} & {\bigwedge}^{0}E^{*} \ar[r] & 0
}$$
which is exact on $U_0$. It is easy to show that the family ${\bigwedge}^{\bullet}(\nabla_{eq}^{*})_{0}=({\bigwedge}^{l}(\nabla_{eq}^{*})_{0},...,{\bigwedge}^{0}(\nabla_{eq}^{*})_{0})$ is compatible with the above sequence on $U_0$. The fact that $ch^{*}({\bigwedge}^{\bullet}(\nabla_{eq}^{*})_{0})=0$ follows from this. Then, we have the localization $\textrm{ch}_{S}^{*}(\lambda_{E^{*}},s)_{eq}$ in $H_{G}^{2i}(M,M\setminus S;\mathbb{C})$, which is represented by the cocycle
$$\textrm{ch}^{*}({\bigwedge}^{\bullet}(\nabla_{eq}^{*})_{\star})=(0, \textrm{ch}^{*}({\bigwedge}^{\bullet}(\nabla_{eq}^{*})_{1}), \textrm{ch}^{*}({\bigwedge}^{\bullet}(\nabla_{eq}^{*})_{0},{\bigwedge}^{\bullet}(\nabla_{eq}^{*})_{1}))$$
We also have the inverse equivariant Todd class $\textrm{td}^{-1}(E)_{eq}$, which is represented by the cocycle
$$\textrm{td}^{-1}((\nabla_{eq})_{\star})=(\textrm{td}^{-1}((\nabla_{eq})_{0}), \textrm{td}^{-1}((\nabla_{eq})_{1}), \textrm{td}^{-1}((\nabla_{eq})_{0},(\nabla_{eq})_{1}))$$
We give some definitions for the theorem in the following. Let $\rho:\mathbb{R}\times U_{01}\to U_{01}$ be the projection and we consider the connection ${\tilde{\nabla}}_{eq}=(1-t)\rho^{*}(\nabla_{eq})_{0}+t\rho^{*}(\nabla_{eq})_{1}$ for $\rho^{*}E$. Let ${\bigwedge}^{\bullet}(\nabla^{*}_{eq})_{\nu}$ denote the family of connections $({\bigwedge}^{l}(\nabla_{eq}^{*})_{\nu},...,{\bigwedge}^{0}(\nabla_{eq}^{*})_{\nu})$ on $U_{\nu}$, for $\nu=0,1$. Also we denote by ${\bigwedge}^{\bullet}({\tilde{\nabla}}^{*}_{eq})$ the family $({\bigwedge}^{l}({\tilde{\nabla}}_{eq}^{*}),...,{\bigwedge}^{0}({\tilde{\nabla}}_{eq}^{*})_{\nu})$. Let $\rho^{\prime}:[0,1]\times U_{01}\to U_{01}$ be the restriction of $\rho$. 

\begin{thm}
In the above situation, we have
$$\textrm{ch}^{*}({\bigwedge}^{\bullet}(\nabla_{eq}^{*})_{\star})=c^{l}((\nabla_{eq})_{\star}))\smile \textrm{td}^{-1}((\nabla_{eq})_{\star}) + D_{eq}\tau$$
where $\tau=(0,0,\tau_{01}), \tau=\rho^{\prime}_{*}(c^{l}(\rho^{*}(\nabla_{eq})_{0}), {\tilde{\nabla}}_{eq})\cdot d_{eq}\textrm{td}^{-1}(\rho^{*}(\nabla_{eq})_{1}), {\tilde{\nabla}}_{eq})$.
\end{thm}

The following corollary follows immediately from this.
\begin{cor}
We have
$$\textrm{ch}_{S}^{*}(\lambda_{E^{*}},s)_{eq}=c_{S}^{l}(E.s)\cdot \textrm{td}^{-1}(E)_{eq}$$
\end{cor}

Also, as an applications of the above, we may get the equivariant universal localized Riemann-Roch theorem for embeddings by using the result in the previous section. Let $\pi:E\to M$ be a $G$-equivariant vector bundle of rank $l$. We have the $G$-equivariant Thom class $\varPsi_{eq}^{E}$ and the Thom isomorphism 
$$T_{E}:H^{*}_{G}(M)\to H^{*+2i}_{G}(E,E\setminus \Sigma)$$
which is given by $T_{E}(\alpha)=\varPsi_{eq}^{E}\cdot \pi^{*}\alpha$. Since $\varPsi_{eq}^{E}=c_{\Sigma}^{l}({\pi^{*}E,s_{\Delta}})_{eq}$, applying the above Corollary to $\pi^{*}E$ and $s_{\Delta}$, we have :

\begin{thm}[Equivariant universal localized RR for embeddings]\label{RR}
\begin{eqnarray*}
\textrm{ch}_{\Sigma}^{*}(\lambda_{\pi^{*}E^{*}},s_{\Delta})_{eq}&=&\varPsi_{eq}^{E}\cdot \textrm{td}^{-1}(\pi^{*}E)_{eq} \\
&=&T_{E}(\textrm{td}^{-1}(E)_{eq})
\end{eqnarray*}
\end{thm}

\end{document}